\documentclass[a4paper,12pt]{amsart}

\usepackage{float}
\usepackage{euscript,eufrak,verbatim}
\usepackage{graphicx}
\usepackage[usenames]{color}
\usepackage[colorlinks,linkcolor=red,anchorcolor=blue,citecolor=blue]{hyperref}
\usepackage{amsmath}
\usepackage{amsthm}
\usepackage[all]{xy}
\usepackage{graphicx} 
\usepackage{amssymb}

\newtheorem{thm}{Theorem}[section]
\newtheorem{prop}[thm]{Proposition}

\newtheorem{lem}[thm]{Lemma}

\newtheorem{ex}[thm]{Example}

\def\N{\mathbb{N}}
\def\Z{\mathbb{Z}}
\def\N{\mathbb{N}}

\def\hh{\underline{h}}

\def\LL{\mathcal{L}}
\def\MM{\mathcal{M}}
\def\XX{\mathcal{X}}
\def\PP{\mathcal{P}}
\def\EE{\mathcal{E}}
\def\UU{\mathcal{U}}
\def\NN{\mathcal{N}}

\def\diam{\text{\rm diam}}
\def\top{\text{\rm top}}
\def\Leb{\text{\rm Leb}}
\def\mdim{\text{\rm mdim}}

\def\mdim{\text{\rm mdim}}

\makeatletter % `@' now normal "letter"
\@addtoreset{equation}{section}
\numberwithin{equation}{section}

\title[On variational principles for metric mean dimension]{On variational principles for metric mean dimension}
\author{Ruxi Shi
}
\address%[authorlabel1]
{Institute of Mathematics, Polish Academy of Sciences, ul. \'Sniadeckich 8, 00-656 Warszawa, Poland}
\email{rshi@impan.pl}
\begin{document}
	
	\maketitle

\begin{abstract}
In this note, we show several variational principles for metric mean dimension. First we prove a variational principles in terms of Shapira's entropy related to finite open covers. Second we establish a variational principle in terms of Katok's entropy.  Finally using these two variational principles we develop a variational principle in terms of Brin-Katok local entropy. %We also discuss variational principles for metric mean dimension in terms of lower Katok's entropy and lower Brin-Katok local entropy respectively.
\end{abstract}

\section{Introduction}

The topological entropy is a basic invariant of dynamical systems which was studied for a long time. The interplay between ergodic theory and topological entropy was first investigated by Goodman \cite{goodman1971relating}.
Gromov \cite{G} introduced a new topological invariant of dynamical systems called mean topological dimension. Mean topological dimension measures the complexity of dynamical systems of infinite entropy. Lindenstrauss and Weiss \cite{LinWeiss2000MeanTopologicalDimension} introduced
metric mean dimension as an invariant of dynamical systems which majors mean topological dimension. The connection between ergodic theory and metric mean dimension was pioneered by Lindenstrauss and Tsukamoto \cite{lindenstrauss2018rate}. They established a variational principle which states metric mean dimension as a supremum of certain rate distortion functions over invariant measures of the system. Very recently, Gutman and \'Spiewak \cite{gutman2020around} showed that it is enough to take supremum over ergodic measures in  Lindenstrauss-Tsukamoto variational principle. For further applications of metric mean dimension, we refer to  \cite{tsukamoto2018large,tsukamoto2018mean,gutman2019new,gutman2020metric} and the references therein.

One of our motivation of this note is Problem 3 in \cite{gutman2020around}, where Gutman and \'Spiewak asked whether metric mean dimension can be expressed in terms of Brin-Katok local  entropy.
In this note, we give an affirmative answer to this problem and consequently build a variational principle for metric mean dimension in terms of Brin-Katok local  entropy. The proof of this variational principle is involved in Section \ref{sec:Variational principle}. To this end, we show a variational principle in terms of Shapira's entropy of an open cover. Furthermore, we prove  an alternative variational principle for metric mean dimension in terms of Katok's entropy, which drops $\lim_{\delta\to 0}$ in \cite{velozo2017rate}. We remark that Brin-Katok local  entropy and  Shapira's entropy are only defined for ergodic measures. Thus the variational principles established in terms of Brin-Katok local  entropy or  Shapira's entropy take the supremum over the set of ergodic measures of the system.

\section{Preliminaries}

\subsection{Topological entropy and variational principle}
Let $(\XX, d, T)$ be a topological dynamical system, i.e. $(\XX, d)$ is a compact metric space and $T: \XX\to \XX$ is a homeomorphism. Define
$$
d_n(x,y)=\max_{0\le k\le n-1} d(T^k x, T^k y),
$$
for $n\in \N$. Let $K\subset X$ and $\epsilon>0$. A subset $E\subset K$ is said to be {\it $(n, \epsilon)$-separated} of $K$ if distinct $x,y\in E$ implies $d_n(x,y)>\epsilon$. Denote by $s_n(d, T, K, \epsilon)$ (simply $s_n(K, \epsilon)$ when $d,T$ are fixed) the largest cardinality of any $(n, \epsilon)$-separated subset of $K$. Define
$$
S(d, T, K,\epsilon)=\limsup_{n\to \infty} \frac{1}{n} \log s_n(d, T, K, \epsilon).
$$
We sometimes write $S(K, \epsilon)$ when $d,T$ are fixed. 

A subset $F\subset \XX$ is said to be {\it $(n, \epsilon)$-spanning} of $K\subset \XX$ if for any $x\in K$ there exists $y\in F$ such that $d_n(x,y)\le \epsilon$. Denote by $r_n(d, T, K, \epsilon)$ (simply $r_n(K, \epsilon)$ when $d,T$ are fixed) the smallest cardinality of any $(n, \epsilon)$-spanning subset in $K$. Define
$$
R(d, T, K,\epsilon)=\limsup_{n\to \infty} \frac{1}{n} \log r_n(d, T, K, \epsilon).
$$
We sometimes write $R(K, \epsilon)$ when $d,T$ are fixed. It is easy to check that
$$
r_n(K,\epsilon)\le s_n(K, \epsilon) \le r_n(K, \frac{\epsilon}{2})
$$
and consequently
$$
R(K,\epsilon)\le S(K, \epsilon) \le R(K, \frac{\epsilon}{2}).
$$

The {\it topological entropy of $K$} is defined by
$$
h_\top(T, K)=\lim\limits_{\epsilon\to 0} S(d, T, K, \epsilon)=\lim\limits_{\epsilon\to 0} R(d, T, K, \epsilon),
$$
which is independent of $d$. 

Let $\mu$ be a $T$-invariant measure on $\XX$, i.e. $T\mu=\mu$. For a Borel partition $P$ of $\XX$, the entropy $H_\mu(P)$ of $P$ is defined by
$$
H_\mu(P)=-\sum_{A\in P} \mu(A)\log \mu(A).
$$
For convention, we set $0\cdot \infty= 0$. Moreover, the (dynamical) entropy of $P$ is defined as 
$$
h_{\mu}(P)=\lim\limits_{n\to \infty} \frac{1}{n} H_{\mu}(\vee_{i=0}^{n-1} T^{-i}P),
$$
where $P\vee Q=\{A\cap B: A\in P, B\in Q \}$. 
The {\it measure-theoretic entropy} $h_{\mu}(T)$ of $\mu$  is defined by
$$
h_\mu(T)=\sup_{P} h_{\mu}(P),
$$
where the suprema are taken over all Borel partitions $P$ of $\XX$.
One of the link between topological entropy and measure-theoretic entropy in ergodic theory  is the variational principle, which was proved originally by Goodman \cite{goodman1971relating}.
\begin{thm}
	Let $(\XX, d, T)$ be a topological dynamical system. Then
	$$
	h(\XX, T)=\sup_{\mu\in\MM_T(\XX)}h_\mu(T)=\sup_{\mu\in\EE_T(\XX)}h_\mu(T),
	$$
	where $\MM_T(\XX)$ is the collection of $T$-invariant measures and $\EE_T(\XX)\subset \MM_T(\XX)$ consists of ergodic ones.
\end{thm}

\subsection{Metric mean dimension}
Let $(\XX, d, T)$ be a topological dynamical system. The {\it upper metric mean dimension} of the system $(\XX, d, T)$ is defined by
$$
\overline{\mdim}_M (\XX, d, T)=\limsup_{\epsilon\to 0} \frac{S(d, T, \XX, \epsilon)}{\log \frac{1}{\epsilon}},
$$
which is also equal to $\limsup_{\epsilon\to 0} \frac{R(d, T, \XX, \epsilon)}{\log \frac{1}{\epsilon}}$.
Similarly, the {\it lower metric mean dimension} is defined by
$$
\underline{\mdim}_M (\XX, d, T)=\liminf_{\epsilon\to 0} \frac{S(d, T, \XX, \epsilon)}{\log \frac{1}{\epsilon}}.
$$
If the upper and lower metric mean dimensions coincide, then we call their common value the metric mean dimension of $(\XX, d, T)$ and denote it by ${\mdim}_M (\XX, d, T)$. Unlike the topological entropy, the metric mean dimension depends on the metric $d$.

Lindenstrauss and Tsukamoto \cite{lindenstrauss2019double} provided a variational principle for metric mean dimension in terms of certain rate-distortion functions. Velozo and Velozo \cite{velozo2017rate} proved an alternative formulation in terms of Katok entropy. Gutman and \'Spiewak \cite{gutman2020around} showed another one in terms of R\'enyi information dimension.
%\begin{thm}\label{thm:gutman}
%	Let $(\XX, d, T)$ be a topological dynamical system. Then
%	$$
%	\overline{\mdim}_M (\XX, d, T)=\limsup_{\epsilon\to 0} \frac{1}{\log \frac{1}{\epsilon}} \sup_{\mu\in \MM_T(\XX)} \%inf_{\text{diam}(P)\le \epsilon}h_\mu (P),
%	$$
%	where the infima are taken over all Borel partitions $P$ of $\XX$ with diameter at most $\epsilon$.
%\end{thm}

\subsection{Entropy of an open cover}
Let $(\XX, d, T)$ be a topological dynamical system. Let $\UU$ be a finite open cover of $\XX$. The {\it topological entropy of $\UU$} is define as 
$$
h_{\top}(\UU, T)=\lim_{n\to \infty} \frac{1}{n} \log \NN(\UU^n), 
$$
where $\UU^n=\vee_{i=0}^{n-1}T^{-i}\UU$ and $\NN(\UU)$ the minimal cardinality of a subcover
of $\UU$.

The following version of the local variational principle for the entropy of an open cover was first conjectured by Romagnoli \cite{romagnoli2003local} and then proved by Glasner and Weiss \cite{GW06}. 
%Furthermore, Huang, Maass, Romagnoli and Ye \cite{huang2004entropy} shown that it suffices to take  supremum over the collection of all ergodic measures instead of all invariant measures in this local variational principle.
\begin{thm}[\cite{GW06}, Theorem 7.11]\label{thm:variant principal open cover}
    Let $(\XX, d, T)$ be a topological dynamical system and let $\UU$ be a finite open cover of $\XX$. Then
    $$
    h_{\top}(\UU, T)=\sup_{\mu\in \MM_T(\XX)} \inf_{\PP\succ \UU} h_{\mu}(\PP, T)=\sup_{\mu\in \EE_T(\XX)} \inf_{\PP\succ \UU} h_{\mu}(\PP, T)
    $$
where the infimum is taken over all finite Borel partitions $\PP$ of $\XX$ which refine $\UU$ (i.e. $A\in \PP$ implies that $A\subset B$ for some $B\in \UU$). 
\end{thm}

Let  $\UU$ be a finite open cover of $\XX$. We denote by $\diam(\UU)$ the diameter of the cover, that is , the maximal diameter of any element of $\UU$. We denote by $\Leb(\UU)$ the Lebesgue number of $\UU$, that is, the largest number $\delta$ with the property that every open ball of radius $\delta$ is contained in an element of $\UU$.
A simple fact which we need is as follow. Indeed it follows by 
$s_n(\XX, 3\text{\rm diam}(\UU))\le \NN(\UU^n)\le s_n(\XX, \text{\rm Leb}(\UU)) $.
See \cite[Lemma 3.5]{gutman2020around} or \cite[Theorem 6.1.8]{downarowicz2011entropy} for details.
\begin{lem}\label{lem:entropy}
    	Let $(\XX, d, T)$ be a topological dynamical system. Let $\UU$ be a finite open cover of $\XX$. Then
    	$$
    	S(\XX, 3\text{\rm diam}(\UU))\le h_{\top}(\UU, T)\le S(\XX, \text{\rm Leb}(\UU)).
    	$$
\end{lem}

Let $\mu\in \EE_T(\XX)$. Let $\UU$ be a finite open cover. For $\delta\in (0,1)$, define $
\NN_\mu(\UU, \delta)
$
to be the minimum number of elements of $\UU$, needed
to cover a subset of $\XX$ whose $\mu$-measure is at least $\delta$. Define
\begin{equation}\label{eq:Shapira}
    h_\mu^S(\UU)=\lim_{n\to \infty} \log \NN_\mu(\UU^n, \delta).
\end{equation}
The above limit exists and is independent of $\delta$ due to Shapira \cite[Theorem 4.2]{shapira2007measure}. Moreover, Shapira proved the following theorem.
\begin{thm}[\cite{shapira2007measure}, Theorem 4.4]\label{thm:shapira}
    Let $(\XX, d, T)$ be a topological dynamical system and let $\UU$ be a finite open cover of $\XX$. Let $\mu\in \EE_T(\XX)$. Then 
    $$
    h_\mu^S(\UU)=\inf_{\PP\succ \UU} h_\mu(\PP), 
    $$
    where $\PP$ runs over all partitions refining $\UU$.
\end{thm}

Combining Theorem \ref{thm:variant principal open cover} and Theorem \ref{thm:shapira}, we have an alternative local variantional principal of a finite open cover as follows: 

\begin{thm}\label{thm:variantional principa open cover shapira}
    Let $(\XX, d, T)$ be a topological dynamical system and let $\UU$ be a finite open cover of $\XX$. Then
\begin{equation*}\label{eq:variant principal shapira}
    h_{\top}(\UU, T)=\sup_{\mu\in \EE_T(\XX)} h_\mu^S(\UU).
\end{equation*}
\end{thm}

We remark that $h_\mu^S(\UU)$ is only defined for ergodic measure $\mu$.

\subsection{Brin-Katok local entropy}
Let $(\XX, d, T)$ be a topological dynamical system. For an invariant measure $\mu\in \MM_T(\XX)$ and a point $x\in \XX$, the {\it Brin-Katok local entropy} at $x$ is defined by
$$
h_\mu^{BK}(x, \epsilon):=\limsup_{n\to \infty} -\frac{1}{n} \log \mu(B_n(x, \epsilon)),
$$
where $B_n(x, \epsilon)=\{y\in \XX: d_n(x,y)<\epsilon \}$. The limit $\lim\limits_{\epsilon\to 0} h_{\mu}^{BK}(x, \epsilon)$ is denoted by $h_\mu^{BK}(x)$.
If additionally $\mu$ is ergodic, then for $\mu$-a.e. $x$,

\begin{itemize}
	\item  $h_{\mu}^{BK}(x, \epsilon)$ is a constant, denoted by $h_\mu^{BK}(\epsilon)$. 
	\item $h_\mu^{BK}(x)=h_\mu(T)$.
\end{itemize}
See \cite{brin1983local} for more details.
%\begin{lem}
%	Let $\mu\in \PP_T(X)$. Suppose $\mu=\int_{\XX} \mu_z d\mu(z)$ is the ergodic decomposition. Then $h_\mu(x, \epsilon)\le \int_{\XX} h_{\mu_z}(x, \epsilon) d\mu(z) $ for all $\epsilon>0$ and $x\in X$.
%\end{lem}
%\begin{proof}
%	By concavity of logarithm, Jensen's inequality and (reverse) Fatou lemma, we obtain that 
%	\begin{equation*}
%	\begin{split}
%	h_{\mu} (x, \epsilon)&=\limsup_{n\to \infty} -\frac{1}{n} \log \int_{\XX} \mu_{z} (B_n(x, \epsilon))d\mu(z)\\
%	&\le \limsup_{n\to \infty} \int_{\XX} -\frac{1}{n} \log \mu_{z}(B_n(x, \epsilon)) d\mu(z)\\
%	&\le  \int_{\XX} \limsup_{n\to \infty} -\frac{1}{n} \log \mu_{z}(B_n(x, \epsilon)) d\mu(z)\\
%	&= \int_{\XX} h_{\mu_z}(x, \epsilon) d\mu(z).
%	\end{split}
%	\end{equation*}
%\end{proof}

Gutman and \'Spiewak showed a lower bound for metric mean dimension in terms of Brin-Katok local entropy and asked whether it is also the upper bound . More precisely, they asked \cite[Problem 3]{gutman2020around}:

\vspace{5pt}

{\it Let $(\XX, d, T)$ be a topological dynamical system. Does the following equality hold? 
	$$\overline{\mdim}_M (\XX, d, T)=\limsup_{\epsilon\to 0} \frac{1}{\log \frac{1}{\epsilon}} \sup_{\mu \in \EE_T(\XX)} h_\mu^{BK} (\epsilon). $$
}
We show a positive answer to this question in Section \ref{sec:Variational principle}.
\subsection{Katok's entropy}
Let $(\XX, d, T)$ be a topological dynamical system. For $\delta\in (0,1)$, $n\in \N$ and $\epsilon>0$, define $N_{\mu}^\delta(n, \epsilon)$ to be  the smallest number of any $(n, \epsilon)$-dynamical balls (i.e. the balls have radius $\epsilon$ in the metric $d_n$) whose union has $\mu$-measure larger
than $\delta$. The {\it Katok's entropy} is defined by
$$
h_{\mu}^K(\epsilon, \delta)=\limsup_{n\to \infty} \frac{1}{n} \log N_\mu^\delta(n, \epsilon).   
$$
Katok \cite{katok1980lyapunov} proved that 
$$
\lim_{\epsilon\to 0} h_{\mu}^K(\epsilon, \delta)=h_\mu(T)
$$
for every $\delta>0$.
\subsection{Local entropy function}
Let $(\XX, d, T)$ be a topological dynamical system. For each $\epsilon>0$ and $x\in X$, we define the {\it local entropy function}
$$
h_d(x, \epsilon)=\inf \{S(K,\epsilon): K~\text{is a closed neighborhood of}~x \},
$$
and
$$
\tilde{h}_d(x, \epsilon)=\inf \{R(K,\epsilon): K~\text{is a closed neighborhood of}~x \},
$$
Clearly, $h_d(x,\epsilon)\ge \tilde{h}_d(x, \epsilon)$. Ye and Zhang showed that \cite[Proposition 4.4]{yezhang2007entropy}, $$\sup_{x\in \XX} \lim\limits_{\epsilon\to 0} h_d(x, \epsilon)=\sup_{x\in \XX} \lim\limits_{\epsilon\to 0} \tilde{h}_d(x, \epsilon)=h_\top(\XX, T).$$

\section{Variational principle I: Shapira's entropy}\label{sec:Variational principle I}

%Let $(\XX, d)$ be a compact metric space. 

Let $F$ be  a $(1, \frac{\epsilon}{4})$-spanning set of $\XX$. Obviously, the finite open cover $\UU:=\{B_1(x,\frac{\epsilon}{2}): x\in F \}$ satisfies that {\rm diam}$(\UU)\le \epsilon$ and {\rm Leb}$(\UU)\ge \frac{\epsilon}{4}$. Then we have the following lemma. See also \cite[Lemma 3.4]{gutman2020around} for the details.
\begin{lem}\label{lem:open cover}
    Let $(\XX, d)$ be a compact metric space. Then for every $\epsilon>0$ there exists a finite open cover $\UU$ of $\XX$ such that {\rm diam}$(\UU)\le \epsilon$ and {\rm Leb}$(\UU)\ge \frac{\epsilon}{4}$. 
\end{lem}

Now we show our first variational principle.
\begin{thm}\label{thm:main shapira}
    Let $(\XX, d, T)$ be a topological dynamical system. Then
	$$\overline{\mdim}_M (\XX, d, T)=\limsup_{\epsilon\to 0} \frac{1}{\log \frac{1}{\epsilon}} \sup_{\mu \in \EE_T(\XX)} \inf_{\diam(\UU)\le \epsilon}h_\mu^S (\UU),$$
	and
	$$\underline{\mdim}_M (\XX, d, T)=\liminf_{\epsilon\to 0} \frac{1}{\log \frac{1}{\epsilon}} \sup_{\mu \in \EE_T(\XX)} \inf_{\diam(\UU)\le \epsilon}h_\mu^S (\UU), $$
	where $\UU$ runs over all finite open covers.
\end{thm}
\begin{proof}
Let $\epsilon>0$.
    By Lemma \ref{lem:open cover}, we can find a finite open cover $\UU_0$ of $\XX$ with $\diam(\UU_0)\le \epsilon$ and $\Leb(\UU_0)\ge \frac{\epsilon}{4}$. Let $\mu\in \EE_T(\XX)$. By Lemma \ref{lem:entropy} and Theorem \ref{thm:variantional principa open cover shapira},  
    \begin{equation*}
	    \begin{split}
	      \sup_{\mu \in \EE_T(\XX)} \inf_{\diam(\UU)\le \epsilon}h_\mu^S (\UU)
	        &\le  \sup_{\mu\in \EE_T(\XX)} h_\mu^S(\UU_0) =h_{\top}(\UU_0, T)\\
	        &\le S(\XX, \text{\rm Leb}(\UU_0))\le S(\XX, \frac{\epsilon}{4}).
	    \end{split}
	\end{equation*}
    
    On the other hand, it is clear that for any finite cover $\UU$ with $\diam{(\UU)}\le \frac{\epsilon}{8}$, the cover $\UU^n$ refines $\UU_0^n$ and as a consequence $\NN_\mu(\UU^n, \delta)\ge \NN_\mu(\UU_0^n, \delta)$ for any $0<\delta<1$. Thus $\inf_{\diam(\UU)\le \frac{\epsilon}{8}}h_\mu^S (\UU)\ge h_\mu^S (\UU_0)$. Then by Lemma \ref{lem:entropy} and Theorem \ref{thm:variantional principa open cover shapira}, we have
    \begin{equation*}
	    \begin{split}
	        	\sup_{\mu\in \EE_T(\XX)}\inf_{\diam(\UU)\le \frac{\epsilon}{8}}h_\mu^S (\UU)
	        &\ge  \sup_{\mu\in \EE_T(\XX)} h_\mu^S(\UU_0) =h_{\top}(\UU_0, T)\\
	        &\ge S(\XX, 3\text{\rm diam}(\UU_0))\ge S(\XX, 3\epsilon).
	    \end{split}
	\end{equation*}
	We conclude that $$S(\XX, 12\epsilon)\le \sup_{\mu\in \EE_T(\XX)}\inf_{\diam(\UU)\le \epsilon}h_\mu^S (\UU)\le S(\XX, \frac{\epsilon}{4})$$ for any $\epsilon>0$. Therefore we complete the proof by the definition of metric mean dimension.
\end{proof}

\section{Variational principle II: Katok's entropy}\label{sec:Variational principle II}
Let $(\XX, d)$ be a compact metric space. For $\delta\in (0,1)$, $n\in \N$ and $\epsilon>0$, define $\widetilde{N}_{\mu}^\delta(n, \epsilon)$ to be  the smallest number of sets with diameter at most $\epsilon$ in the metric $d_n$ whose union has $\mu$-measure larger than $\delta$. Recall that $N_{\mu}^\delta(n, \epsilon)$ is the smallest number of any $(n, \epsilon)$-dynamical balls whose union has $\mu$-measure larger
than $\delta$.  Clearly,
\begin{equation}\label{eq：delta}
    \widetilde{N}_{\mu}^\delta(n, 2\epsilon) \le {N}_{\mu}^\delta(n, \epsilon) \le  \widetilde{N}_{\mu}^\delta(n, \epsilon).
\end{equation}
\begin{lem}\label{lem:delta covering number}
    Let $(\XX, d, T)$ be a topological dynamical system.  Let $\mu$ be an ergodic measure. Let $\UU$ be a finite open cover of $\XX$ with $\diam(\UU)\le \epsilon_1$ and $\Leb(\UU)\ge \epsilon_2$. Let $\delta\in (0,1)$. Then
    $$
    \widetilde{N}_{\mu}^\delta(n, \epsilon_1) \le \NN_\mu(\UU^n, \delta)\le {N}_{\mu}^\delta(n, \epsilon_2).
    $$
\end{lem}
\begin{proof}
    The inclusion  $\widetilde{N}_{\mu}^\delta(n, \epsilon_1) \le \NN_\mu(\UU^n, \delta)$ is trivial. Let $F$ be a collection of $(n, \epsilon_2)$-dynamical balls with $\sharp F={N}_{\mu}^\delta(n, \epsilon_2)$ whose union has $\mu$-measure larger than $\delta$. Then for each $B\in F$, there is $U_B\in \UU^n$ such that $B\in U_B$. Then the union of $U_B, B\in F$, has $\mu$-measure larger than $\delta$. Thus $\NN_\mu(\UU^n, \delta)\le {N}_{\mu}^\delta(n, \epsilon_2).$
\end{proof}

Our second result on variational principle is as follow.
\begin{thm}\label{thm:main Katok}
    Let $(\XX, d, T)$ be a topological dynamical system. Then for every $\delta\in (0,1)$,
	$$\overline{\mdim}_M (\XX, d, T)=\limsup_{\epsilon\to 0} \frac{1}{\log \frac{1}{\epsilon}} \sup_{\mu \in \EE_T(\XX)} h_\mu^K (\epsilon, \delta),$$
	and
	$$\underline{\mdim}_M (\XX, d, T)=\liminf_{\epsilon\to 0} \frac{1}{\log \frac{1}{\epsilon}} \sup_{\mu \in \EE_T(\XX)} h_\mu^K (\epsilon, \delta). $$
\end{thm}
\begin{proof}
    By Lemma \ref{lem:open cover}, we can find a finite open cover $\UU$ of $\XX$ with $\diam(\UU)\le \epsilon$ and $\Leb(\UU)\ge \frac{\epsilon}{4}$. Fix $\delta\in (0,1)$. Let $\mu\in \EE_T(\XX)$. Let $\sigma>0$. By Lemma \ref{lem:delta covering number} and \eqref{eq：delta}, we have 
    \begin{equation*}
        {N}_{\mu}^\delta(n, \epsilon) \le \NN_\mu(\UU^n, \delta) \le {N}_{\mu}^\delta(n, \frac{\epsilon}{4}).
    \end{equation*}
    It follows that
    \begin{equation}\label{eq:Katok and Shapira}
        h_{\mu}^{K}(\epsilon, \delta) \le h_\mu^S(\UU) \le h_{\mu}^{K}(\frac{\epsilon}{4}, \delta).
    \end{equation}
   Combining this with Lemma \ref{lem:entropy} and Theorem \ref{thm:variantional principa open cover shapira}, we have
	\begin{equation*}
	    \begin{split}
	\sup_{\mu\in \EE_T(\XX)}h_{\mu}^{K}(\epsilon, \delta)
	&\le \sup_{\mu\in \EE_T(\XX)} h_\mu^S(\UU)=h_{\top}(\UU, T)\\
	&\le S(\XX, \text{\rm Leb}(\UU))\le S(\XX, \frac{\epsilon}{4}).
	\end{split}
	\end{equation*}
    Similarly,
    \begin{equation*}
	    \begin{split}
	        	\sup_{\mu\in \EE_T(\XX)}h_{\mu}^{K}(\frac{\epsilon}{4}, \delta)
	        &\ge  \sup_{\mu\in \EE_T(\XX)} h_\mu^S(\UU) =h_{\top}(\UU, T)\\
	        &\ge S(\XX, 3\text{\rm diam}(\UU))\ge S(\XX, 3\epsilon).
	    \end{split}
	\end{equation*}
	We conclude that $S(\XX, 12\epsilon)\le \sup_{\mu\in \EE_T(\XX)}h_{\mu}^{K}(\epsilon, \delta)\le S(\XX, \frac{\epsilon}{4})$ for any $\epsilon>0$ and $0<\delta<1$. Therefore we complete the proof by the definition of metric mean dimension.
\end{proof}

\section{Variational principle III: Brin-Katok entropy}\label{sec:Variational principle}

In this section, we show the variational principle for metric mean dimension in terms of Brin-Katok local entropy, which also gives a positive answer to  Problem 3 in \cite{gutman2020around}.

Let $(\XX, d, T)$ be a topological dynamical system. For a cover $\UU$ of $\XX$ and $\mu\in \MM_T(\XX)$, we define
$$
h_\mu^{BK}(x, \UU):=\limsup_{n\to \infty} -\frac{1}{n} \log \mu(\UU_x^n),
$$
where $\UU_x^n$ is the union of all cells of the cover $\UU^n$ which contain $x$. If additionally $\mu$ is ergodic, then  $h_{\mu}^{BK}(x, \UU)$ is a constant for $\mu$-a.e. $x$, denoted by $h_\mu^{BK}(\UU)$. Moreover, if $\UU$ is a partition, then $h_\mu^{BK}(\UU)=h_\mu(\UU)$ by Shannon-McMillan-Breiman theorem

\begin{lem}\label{lem:local and cover}
    Let $(\XX, d, T)$ be a topological dynamical system. Let $\UU$ be a finite cover of $\XX$. Let $\epsilon_1, \epsilon_2>0$. Suppose diam$(\UU)\le \epsilon_1$ and Leb$(\UU)\ge \epsilon_2$. Then 
    $$
    h_{\mu}^{BK}(\epsilon_1)\le h_{\mu}^{BK}(\UU) \le h_{\mu}^{BK}(\epsilon_2),    
    $$
    for any $\mu\in \EE_T(\XX)$.
\end{lem}
\begin{proof}
    It follows by the inclusion $B_n(x, \epsilon_2)\subset \UU^n_x\subset B_n(x,\epsilon_1)$ for every $n\in \N$ and $x\in \XX$. 
\end{proof}

Now we present our main result in this section.
\begin{thm}\label{thm:main 1}
	Let $(\XX, d, T)$ be a topological dynamical system. Then
	$$\overline{\mdim}_M (\XX, d, T)=\limsup_{\epsilon\to 0} \frac{1}{\log \frac{1}{\epsilon}} \sup_{\mu \in \EE_T(\XX)} h_\mu^{BK} (\epsilon),$$
	and
	$$\underline{\mdim}_M (\XX, d, T)=\liminf_{\epsilon\to 0} \frac{1}{\log \frac{1}{\epsilon}} \sup_{\mu \in \EE_T(\XX)} h_\mu^{BK} (\epsilon). $$
\end{thm}
\begin{proof}
	Fix $\epsilon>0$. By Lemma \ref{lem:open cover}, we can find a finite open cover $\UU$ of $\XX$ with $\diam(\UU)\le \epsilon$ and $\Leb(\UU)\ge \frac{\epsilon}{4}$. Since $\diam(\PP)\le \epsilon$ for any partition $\PP\succ \UU$, by Lemma \ref{lem:local and cover}, we have $h_{\mu}^{BK}(\epsilon)\le \inf_{\PP\succ \UU} h_{\mu}(\PP)$ for any $\mu\in \EE_T(\XX)$. Then by Lemma \ref{lem:entropy} and Theorem \ref{thm:variant principal open cover}
	\begin{equation*}
	    \begin{split}
	\sup_{\mu\in \EE_T(\XX)}h_{\mu}^{BK}(\epsilon)
	&\le \sup_{\mu\in \EE_T(\XX)} \inf_{\PP\succ \UU} h_{\mu}(\PP)=h_{\top}(\UU, T)\\
	&\le S(\XX, \text{\rm Leb}(\UU))\le S(\XX, \frac{\epsilon}{4}).
	\end{split}
	\end{equation*}
	This implies LHS $\ge$ RHS.

	It remains to show LHS $\le$ RHS. Let $\mu\in \EE_T(\XX)$. Let $\sigma>0$ and let
    $$
    G_{n, \sigma}=\{x\in \XX: -\frac{1}{n}\log \mu(B_n(x,\epsilon))<h_{\mu}^{BK}(\epsilon)+\sigma \}.
    $$
    Since $\mu(\cup_{N\ge 1} \cap_{n\ge N}G_{n,\sigma})=1$ and $\cap_{n\ge N}G_{n,\sigma}$ is increasing as $N$ grows, we have
    $
    \lim_{N\to \infty}\mu(\cap_{n\ge N}G_{n,\sigma})=1.
    $
    Let $\delta\in (0,1). $ Then there exists $n_0\in \N$ such that for any $n\ge n_0$, $\mu(G_{n, \sigma})>\delta$. Pick arbitrary $n\ge n_0$. Let $H_n$ be a maximal $(n,2\epsilon)$-separated set of $G_{n, \sigma}$. It follows that $H_n$ is a $(n, 2\epsilon)$-spanning set of $G_{n, \sigma}$. Thus the union of the balls $B_n(x, 3\epsilon), x\in H_n,$ cover $G_{n,\sigma}$. It implies that
    $$
    \mu(\cup_{x\in H_n }B_n(x, 3\epsilon))\ge \mu(G_{n, \sigma})>\delta.
    $$
    That is to say, $\sharp H_n\ge N_{\mu}^\delta(n, 3\epsilon)$. On the other hand, since $H_n$ is the $(n,2\epsilon)$-separated set, the balls $B_n(x, \epsilon), x\in H_n,$ are disjoint. It follows that
    $$
    1\ge \mu(\cup_{x\in H_n }B_n(x, \epsilon))=\sum_{x\in H_n}\mu(B_n(x, \epsilon))\ge \sharp H_n e^{-n(h_{\mu}^{BK}(\epsilon)+\sigma)},
    $$
    where the last inequality is due to the fact that $H_n\subset G_{n,\sigma}$. Then 
    $
    \sharp H_n \le e^{n(h_{\mu}^{BK}(\epsilon)+\sigma)}.
   $
    Therefore we get 
    $$
    N_{\mu}^\delta(n, 3\epsilon)\le e^{n(h_{\mu}^{BK}(\epsilon)+\sigma)},
    $$
	and consequently 
	$$
	h_\mu^K(3\epsilon, \delta)\le h_\mu^{BK}(\epsilon)+\sigma.
	$$
	Since $\sigma$ is chosen arbitrarily, by Theorem \ref{thm:main Katok}, this completes the proof.
\end{proof}

\begin{ex}
	Let $\XX=[0,1]^\Z$ be the infinite product of the unit interval. Let $\sigma: \XX \to \XX$ be the (left) shift defined by $(x_n)_{n\in \Z} \mapsto (x_{n+1})_{n\in \Z}$. Define a distance $d$ on $\XX$ by
	$$
	d(x,y)=\sum_{n\in \Z} 2^{-|n|} |x_n-y_n|.
	$$
	It is known that $\mdim_M ([0,1]^\Z, d, \sigma )=1$ (see for instance \cite[Example 1.1]{lindenstrauss2019double}). Let $\LL$ be the Lebesgue measure on $[0,1]$ and $\mu=\LL^{\otimes \Z}$. We will calculate $h_{\mu}^{BK}(\epsilon)$ for $\epsilon>0$. 
	
	Let $\epsilon>0$ and $x\in [0,1]^\Z$. Set $\ell=\lceil \log_2 \frac{4}{\epsilon} \rceil$. Then $\sum_{|n|>\ell} 2^{-|n|}\le \epsilon/2$.
	Let 
	$$
	I_n(x, \epsilon):=\{y\in [0,1]^\Z: y_{k}\in x_{k}+[-\frac{\epsilon}{6}, \frac{\epsilon}{6}], \forall -\ell\le k\le n+\ell \},
	$$
	and 
	$$
	J_n(x, \epsilon):=\{y\in [0,1]^\Z: y_{k}\in x_{k}+[-\epsilon, \epsilon], \forall 0\le k\le n\}.
	$$
	 It is easy to see that
	$$
	I_n(x, \epsilon) \subset B_n(x, \epsilon)\subset J_n(x, \epsilon).
	$$
	Since $\mu(I_n(x, \epsilon))\ge \left( \frac{\epsilon}{6} \right)^{n+\ell}$ and $\mu(J_n(x, \epsilon))\le \left( 4\epsilon \right)^{n}$, we obtain that
	$$
	\log \frac{1}{4\epsilon} \le h_\mu^{BK}(\epsilon)\le \log \frac{3}{\epsilon}.
	$$
	Therefore
	$$
	\lim\limits_{\epsilon\to 0} \frac{h_\mu^{BK}(\epsilon)}{\log \frac{1}{\epsilon}} =1 =\mdim_M ([0,1]^\Z, d, \sigma ).
	$$
\end{ex}

\section{Discussion on lower Brin-Katok local entropy and lower Katok's entropy}
Let $(\XX, d, T)$ be a topological dynamical system. For an invariant measure $\mu\in \MM_T(\XX)$ and a point $x\in \XX$, analogous to Brin-Katok entropy, we define the {\it lower Brin-Katok local entropy} by
$$
\underline{h}_\mu^{BK}(x, \epsilon):=\liminf_{n\to \infty} -\frac{1}{n} \log \mu(B_n(x, \epsilon)).
$$
The limit $\lim\limits_{\epsilon\to 0} \underline{h}_\mu^{BK}(x, \epsilon)$ is denoted by $\underline{h}_\mu^{BK}(x)$.
If additionally $\mu$ is ergodic, then $\underline{h}_\mu^{BK}(x, \epsilon)$ is a constant for $\mu$-a.e. $x$, denoted by $\underline{h}_\mu^{BK}(\epsilon)$ and as a consequence $\underline{h}_\mu^{BK}(x)=h_\mu(T)$. 

Recall that $N_{\mu}^\delta(n, \epsilon)$ is the smallest number of any $(n, \epsilon)$-dynamical balls whose union has $\mu$-measure larger
than $\delta$. Analogous to Katok's entropy, we define the {\it lower Katok's entropy} by
$$
\underline{h}_{\mu}^K(\epsilon, \delta)=\liminf_{n\to \infty} \frac{1}{n} \log N_\mu^\delta(n, \epsilon).   
$$

 For a cover $\UU$ of $\XX$ and $\mu\in \MM_T(\XX)$, we define
$$
\underline{h}_\mu(x, \UU):=\liminf_{n\to \infty} -\frac{1}{n} \log \mu(\UU_x^n).
$$
 If additionally $\mu$ is ergodic, then  $\underline{h}_{\mu}(x, \UU)$ is a constant for $\mu$-a.e. $x$, denoted by $\underline{h}_\mu(\UU)$. 
Same as Lemma \ref{lem:local and cover}, we have the following lemma.
\begin{lem}\label{lem:local and cover 2}
    Let $(\XX, d, T)$ be a topological dynamical system. Let $\UU$ be a finite open cover of $\XX$. Let $\epsilon_1, \epsilon_2>0$. Suppose that diam$(\UU)\le \epsilon_1$ and Leb$(\UU)\ge \epsilon_2$. Then 
    $$
    \hh_{\mu}^{BK}(\epsilon_1)\le \hh_{\mu}(\UU) \le \hh_{\mu}^{BK}(\epsilon_2), 
    $$
    for any $\mu\in \EE_T(\XX)$. 
\end{lem}

In the proof of Theorem \ref{thm:main 1} , we see that $h_\mu^K$ is bounded from above by $h_\mu^{BK}$. We show in the following proposition that $h_\mu^K$ is bounded from below by $\hh_\mu^{BK}$.

\begin{prop}\label{prop:1}
    Let $(\XX, d, T)$ be a topological dynamical system. Then for every $\delta\in (0,1)$ and $\epsilon>0$,
	$$
	h_\mu^K(\frac{\epsilon}{4}, \delta)\ge \hh_\mu^{BK}(\epsilon),  \forall \mu\in \EE_T(\XX).
	$$
\end{prop}
\begin{proof}
    By Lemma \ref{lem:open cover}, we can find a finite open cover $\UU$ of $\XX$ with $\diam(\UU)\le \epsilon$ and $\Leb(\UU)\ge \frac{\epsilon}{4}$. Fix $\delta\in (0,1)$. Let $\mu\in \EE_T(\XX)$. Let $\sigma>0$. Let 
    $$
    F_{n, \sigma}=\{x\in \XX: -\frac{1}{n}\log \mu(\UU_x^n)>\hh_{\mu}(\UU)-\sigma \}.
    $$
    Since $\mu(\cup_{N\ge 1} \cap_{n\ge N}F_{n,\sigma})=1$ and $\cap_{n\ge N}F_{n,\sigma}$ is increasing as $N$ grows, we have
    $
    \lim_{N\to \infty}\mu(\cap_{n\ge N}F_{n,\sigma})=1.
    $
    Then there exists $n_0\in \N$ such that for any $n\ge n_0$, $\mu(F_{n, \sigma})>1-\frac{\delta}{2}$. Since $\Leb(\UU)\ge \frac{\epsilon}{4}$, we see that a $(n, \frac{\epsilon}{4})$-dynamical ball containing a point $x\in F_{n, \sigma}$ is entirely
contained in $\UU_x^n$, so its measure is at most $e^{-n(\hh_{\mu}(\UU)-\sigma)}$. For $n>n_0$, note that the $\mu$-measure of the intersection between the complement of $F_{n, \sigma}$ and any union of $(n, \frac{\epsilon}{4})$-dynamical balls in $\XX$ whose
measure larger than $\delta$ is smaller or equal to $\delta/2$. Thus
$$
N_{\mu}^{\delta}(n, \frac{\epsilon}{4})\ge \frac{\delta}{2} e^{n\hh_{\mu}(\UU)-n\sigma}, \forall n>n_0.
$$
It follows that $h_\mu^K(\frac{\epsilon}{4}, \delta)\ge \hh_\mu(\UU)-\sigma$. Since $\sigma$ is arbitrary, by Lemma \ref{lem:local and cover 2} we get 
$$h_\mu^K(\frac{\epsilon}{4}, \delta)\ge \hh_\mu(\UU)\ge \hh_\mu^{BK}(\epsilon).$$ 
This completes the proof.
\end{proof}

By Lemma \ref{lem:delta covering number} and the fact that limit \eqref{eq:Shapira} exists, it follows from the same proof of \eqref{eq:Katok and Shapira} that 
\begin{equation}\label{eq:lower}
        \hh_{\mu}^{K}(\epsilon, \delta) \le h_\mu^S(\UU) \le \hh_{\mu}^{K}(\frac{\epsilon}{4}, \delta).
    \end{equation}
Using \eqref{eq:lower} and the same proof of Theorem \ref{thm:main Katok}, we get the variational principle for metric mean dimension hold in terms of $\hh_\mu^K$. We omit the proof here and leave it to the readers to work out the details. 
\begin{prop}\label{prop:2}
    Let $(\XX, d, T)$ be a topological dynamical system. Then for every $\delta\in (0,1)$,
	$$\overline{\mdim}_M (\XX, d, T)=\limsup_{\epsilon\to 0} \frac{1}{\log \frac{1}{\epsilon}} \sup_{\mu \in \EE_T(\XX)} \hh_\mu^K (\epsilon, \delta),$$
	and
	$$\underline{\mdim}_M (\XX, d, T)=\liminf_{\epsilon\to 0} \frac{1}{\log \frac{1}{\epsilon}} \sup_{\mu \in \EE_T(\XX)} \hh_\mu^K (\epsilon, \delta). $$
\end{prop}

We end up this section with open problems as follows.

\text{\bf Problem 1}: Does the variational principle for metric mean dimension hold in terms of $\hh_\mu^K$?

\text{\bf Problem 2}: Is $\hh_\mu^{BK}$  bounded from below by $\hh_\mu^{K}$? That is to say, does there exist a constant $c$ such that for every $\delta\in (0,1)$ and $\epsilon>0$, one has that
	$
	\hh_\mu^K(c\epsilon, \delta)\le \hh_\mu^{BK}(\epsilon), \forall \mu\in \EE_T(\XX)?$

\vspace{5pt}

We remark that the affirmative answer to  Problem 2 will give a positive answer to Problem 1.

\section{On local entropy function}\label{sec:Theorem II}

In this section, we show that the metric mean dimension is related to the local  entropy function. Tsukamoto \cite[Lemma 2.5]{tsukamoto2018mean} showed a formula of metric mean dimension in terms of local quantity. We develop an alternative formula in terms of local entropy function.

\begin{thm}
	Let $(\XX, d, T)$ be a topological dynamical system. Then
	$$\overline{\mdim}_M (\XX, d, T)=\limsup_{\epsilon\to 0} \frac{1}{\log \frac{1}{\epsilon}} \sup_{x\in \XX} h_d (x,\epsilon). $$
\end{thm}
\begin{proof}
	Since $h_d(x,\epsilon)\le S(\XX, \epsilon)$ for all $x\in \XX$, it is obvious that LHS $\ge$ RHS. The other inequality follows from Lemma \ref{lem:local entropy} and $h_d(x,\epsilon)\ge \tilde{h}_d(x,\epsilon)$.
\end{proof}
By same argument, we also have that $$\underline{\mdim}_M (\XX, d, T)=\liminf_{\epsilon\to 0} \frac{1}{\log \frac{1}{\epsilon}} \sup_{x\in \XX} h_d (x,\epsilon). $$ 
\begin{lem}\label{lem:local entropy}
	Let $K$ be a closed susbet of $X$. Then $\sup_{x\in K}\tilde{h}_d(x, \epsilon)\ge R(K, \epsilon)$.
\end{lem}
\begin{proof}
	Let $\{B_1^1, B_2^1, \dots, B_{n_1}^1 \}$ be a cover of $K$ consisting of closed balls with diameter at most $1$. Then there exists $j_1$ such that
	$$
	R(K, \epsilon)=R(B_{j_1}^1\cap K, \epsilon). 
	$$
	Cover $B_{j_1}^1\cap K$ by closed balls $B_1^2, B_2^2, \dots, B_{n_2}^2$ with diameter at most $\frac{1}{2}$. Then there exists $j_2$ such that
	$$
	R(K, \epsilon)=R(B_{j_2}^1\cap K, \epsilon). 
	$$
	By induction, for every $k\ge 2$, there  exists a closed ball $B_{j_k}^k$ with diameter at most $\frac{1}{k}$ such that 
	$$
	R(K, \epsilon)=R(B_{j_k}^k\cap K, \epsilon). 
	$$
	Let $\bar{x}=\cap_{k\in \N} B_{j_k}^k$ (which is equal to $\cap_{k\in \N} (B_{j_k}^k\cap K)$ by above construction). For any closed neighborhood $K'$ of $\bar{x}$, we can find sufficiently large $k\in \N$ such that $B_{j_k}^k\cap K\subset K'$, which implies that
	$$
	R(K',\epsilon)\ge R(B_{j_k}^k\cap K)=R(K,\epsilon),
	$$ 
	that is, $h_d(\bar{x}, \epsilon)\ge R(K,\epsilon)$. This completes the proof. 
	
\end{proof}

\section*{Acknowledgement} 
We thank Adam {\'S}piewak for pointing out a mistake  on an earlier draft. %We would like to thank Bingbing Liang for showing us the validation of Lemma \ref{lem:3}. 
We are grateful to Yonatan Gutman for valuable remarks.

\bibliographystyle{alpha}
\bibliography{universal_bib}

\end{document}